\documentclass[11pt]{amsart}

\usepackage{verbatim}
\usepackage{amssymb,amsthm,amsmath}
\usepackage{color}
\definecolor{darkblue}{rgb}{0, 0, .4}
\definecolor{grey}{rgb}{.7, .7, .7}

\usepackage{ifpdf}
\ifpdf
  \usepackage[pdftex]{epsfig}
  \usepackage{lscape}
  \usepackage[breaklinks]{hyperref}
  \hypersetup{
            colorlinks=true,
            linkcolor=darkblue,
            anchorcolor=darkblue,
            citecolor=darkblue,
            pagecolor=darkblue,
            urlcolor=darkblue,
            pdftitle={},
            pdfauthor={}
  }
\else
  \usepackage[dvips]{epsfig}
  \usepackage{lscape}
  \newcommand{\href}[2]{#2}
  \newcommand{\url}[2]{#2}
\fi

\newtheorem{theorem}{Theorem}[section]
\newtheorem{lemma}[theorem]{Lemma}

\theoremstyle{definition}

\newtheorem{example}[theorem]{Example}

\theoremstyle{remark}

\numberwithin{equation}{section}

\theoremstyle{theorem}

\input xy 
\xyoption{all} 

\usepackage{graphicx}

\newcommand{\abacus}[1]{{ \tiny \xymatrix @-2.3pc { #1 } }}

\newcommand{\ci}[1]{{\xy*{ #1 }*\cir<10pt>{}\endxy}}  
\newcommand{\nc}[1]{{\xy*{ #1 }*i\cir<10pt>{}\endxy}}  

\begin{document}

\title{The Refined Lecture Hall Theorem via Abacus Diagrams}
\date{}

\begin{abstract}
Bousquet-M\'elou \& Eriksson's lecture hall theorem generalizes Euler's
celebrated distinct-odd partition theorem. We present an elementary and
transparent proof of a refined version of the lecture hall theorem using a
simple bijection involving abacus diagrams.
\end{abstract}

\author{Laura Bradford}

\author{Meredith Harris}

\author{Brant Jones}

\author{Alex Komarinski}

\author{Carly Matson}

\author{Edwin O'Shea}
\address{Department of Mathematics and Statistics, MSC 1911, James Madison University, Harrisonburg, VA 22807}
\email{\href{mailto:[jones3bc,osheaem]@jmu.edu}{\texttt{[jones3bc,osheaem]@jmu.edu}}}
\thanks{The authors received support from NSF grant DMS-1004516.}

\maketitle

\bigskip
\section{Introduction} \label{s:intro}

Lecture hall partitions were introduced by  Bousquet-M\'elou and Eriksson \cite{BME} as sequences of non-negative 
integers 
$\lambda = (\lambda_1, \lambda_2, \ldots, \lambda_n)$ satisfying
\begin{equation}\label{e:lh}
0 \leq \frac{\lambda_1}{1} \leq \frac{\lambda_2}{2} \leq \cdots \leq \frac{\lambda_i}{i} \leq \cdots \leq \frac{\lambda_n}{n}.
\end{equation}
Pictorially, the diagram of $\lambda$ represents the heights of seats in a
lecture hall with $n$ rows.  The requirement that each row be able to see the
speaker (who is located at height zero) then corresponds to the slope condition
given in the definition. In \cite{BME}, the following remarkable theorem was shown. 

\begin{theorem} \label{t:main}
\textup{\bf (The Lecture Hall Theorem)}
\ \ We have 
\begin{equation}\label{e:gt}
\displaystyle 
\sum_{\lambda}x^{|\lambda|} 
\, = \, 
\prod_{i=1}^{n}\frac{1}{1-x^{2i-1}} \, 
\end{equation}
where the sum is taken over all lecture hall partitions $\lambda$ with $n$ parts and $|\lambda| = \sum_{i=1}^n \lambda_i$. 
\end{theorem}

This can be viewed as a finite generalization of Euler's classical result that
the number of partitions of a given integer having distinct parts is equal to
the number of partitions of that integer having odd parts.  To see this, observe
that the lecture hall inequalities (\ref{e:lh}) imply that $\lambda$ always has
distinct parts.  Conversely, if we are given any partition $\lambda$ with distinct
parts, then there exists an $N$ for which the partitions of length $n > N$ obtained
from $\lambda$ by including parts of size zero all satisfy the lecture hall
inequalities.  In this sense, the left side of (\ref{e:gt}) becomes the
generating function for partitions with distinct parts as $n \rightarrow
\infty$, while the right side of (\ref{e:gt}) becomes the generating function
for partitions with odd parts.  Hence, we recover Euler's result.  A gentle
introduction to other generalizations of Euler's result can be found in
\cite[Chapter 9]{andrews-eriksson-book}.

Bousquet-M\'{e}lou and Eriksson gave two proofs of Theorem~\ref{t:main} in
\cite{BME}:   one relied on Bott's formula for the affine Weyl group
$\widetilde{C}$ and the other was a relatively complicated recursive argument.
Shortly thereafter they further refined Theorem~\ref{t:main} and gave the first
truly bijective proof \cite[\S 3]{BME3} of the Lecture Hall Theorem.  Our
bijection also provides a proof of this refined version of the Lecture Hall
Theorem; see Theorem~\ref{t:refined}.  Other bijective proofs followed by Yee
\cite{yee-2001, yee-2002} which also proved the refined version, and by Eriksen
\cite{eriksen} whose construction gave further support to some open conjectures
on generalized lecture hall partitions \cite{BME2}. Savage and Yee
\cite{savage-yee} also gave a new proof by studying the more general
$\ell$-sequences.  These bijective proofs are elementary yet also somewhat
involved. 

Our proof adds to this body of work by providing a bijection that is both
elementary and straightforward.  The proof boils down to three pictures (see
Examples~\ref{ex:1}, \ref{ex:2}, and \ref{ex:3}) involving abacus diagrams, each
of which are simple and intuitive.  It remains to be seen if our streamlined
proof lends itself to the generalized versions of Theorem~\ref{t:main}. 

Many proofs of Theorem~\ref{t:main} are short yet rely heavily on background
knowledge of some external theory. For example, the other proof of \cite{BME3}
relied on a $q$-analog of Bott's formula for $\widetilde{C}$ and MacMahon's
partition analysis was utilized by Andrews in \cite{andrews}. Other proofs by
Savage et al. use $q$-series \cite{corteel-savage,acs}. There has also been
extensive work done by Savage and others \cite{savage-five} on understanding
the geometry of lecture hall partitions as lattice points in the cone given by
the inequalities that define those partitions. The recursive proof in
\cite{BME} can be interpreted in these geometric terms but an honest proof of
the lecture hall theorem in this lattice point sense currently remains out of
reach.  

In this article, we will develop abacus diagrams from scratch as a natural way
to encode lecture hall partitions.  Abacus diagrams were originally introduced
by James \cite{james-kerber} to study the modular representation theory of the
finite symmetric group.  These ``type $A$'' diagrams correspond to core
partitions and have been used by Wildon \cite{wildon} and Garvin--Kim--Stanton
\cite{garvin-kim-stanton} to study the partition function.  The abacus diagrams
in our work have appeared previously \cite{hanusa-jones} as minimal length coset
representatives in the affine Weyl group $\widetilde{C}$, and correspond to
symmetric core partitions.  We do not rely on these connections in our work.

Sections~\ref{s:abaci_lecturehalls} through \ref{s:bijpres} constitute our proof of Theorem~\ref{t:main}.  In
Section~\ref{s:abaci_lecturehalls} we explain how to encode lecture hall
partitions as abacus diagrams.  In Section~\ref{s:abbound}, we show that the
abacus diagrams are also in bijection with certain partitions whose parts are
bounded.  (This fact was shown previously in \cite{hanusa-jones}, but we include
a proof here to be self-contained.)  It is straightforward to verify that the
generating function for these bounded partitions is the same one that appears in
the Lecture Hall Theorem.  We show in Section~\ref{s:bijpres} that the
composite bijection from lecture hall partitions to bounded partitions preserves
the sum-of-parts statistic.  This shows that the lecture hall partitions have
the same generating function as the bounded partitions, and completes the proof
of the Lecture Hall Theorem.  In Section~\ref{s:refined} we prove the refined
version of the Lecture Hall Theorem that is given in Theorem~\ref{t:refined}.
Finally, in Section~\ref{s:conclusions}, we conclude with some remarks
indicating connections to the Coxeter group of type $\widetilde{C}$.

\bigskip
\section{Abacus diagrams for lecture hall partitions} \label{s:abaci_lecturehalls}

Fix a positive integer $n$.  In our work, we use a particular type of diagram to
encode the lecture hall partitions of length $n$, which we now describe.  We
begin with an array having $2n$ columns and countably many rows.  We label the
entry in the $i$th row and $j$th column of the array by the integer $j + 2ni$,
where $1 \leq j \leq 2n$.  In figures, we will draw the rows increasingly down
the page, and columns increasingly from left to right.  Then these labels
linearly order the entries of the array, which we refer to as {\bf reading
order}.  We also say that column $j$ is {\bf dual} to column $2n+1-j$, and we
call the entries $\{1+(k-1)n, 2+(k-1)n, \ldots, (n-1)+(k-1)n, nk\}$ the $k$th {\bf
window} of the array.  To create our diagram, we highlight certain entries in
the array; such entries are called {\bf beads} and will be circled in figures.
Entries that are not beads will be called {\bf gaps}.

To encode a lecture hall partition $\lambda$, we begin with the largest part
$\lambda_n$, and set entry $\lambda_n$ in the array to be a bead $b_n$.  Next,
skipping entries that lie in the column containing $b_n$ or its dual column, we
count out $\lambda_{n-1}$ positive positions and place a bead $b_{n-1}$.
Continuing in this way, we place one bead $b_i$ for each part $\lambda_i$ by
counting out $\lambda_i$ positive entries, not including the entries of any
column containing a previously placed bead $b_j$ for $j>i$, nor the duals of
such columns.  If $\lambda_i = 0$, then a bead $b_i$ is placed at the largest
nonpositive entry in a column that does not contain a previously placed bead,
nor the dual of a column containing a previously placed bead.  We will refer to
these beads $b_i$ as {\bf defining beads}.

In order to complete the diagram, we perform two additional steps for each
defining bead $b_i$.  First, we create beads at all of the entries above $b_i$
lying in the same column as $b_i$.  Second, if $b_i$ has label $j$, then the
entry labeled $1-2n-j$ occurs in the dual column to $b_i$.  We create beads at
this entry, and all entries lying above it in the dual column to $b_i$.  All of
the other entries in the diagram are gaps.  We call this completed diagram the
{\bf abacus diagram for $\lambda$}.

\bigskip
\begin{example}\label{ex:1}
\begin{figure}[ht]
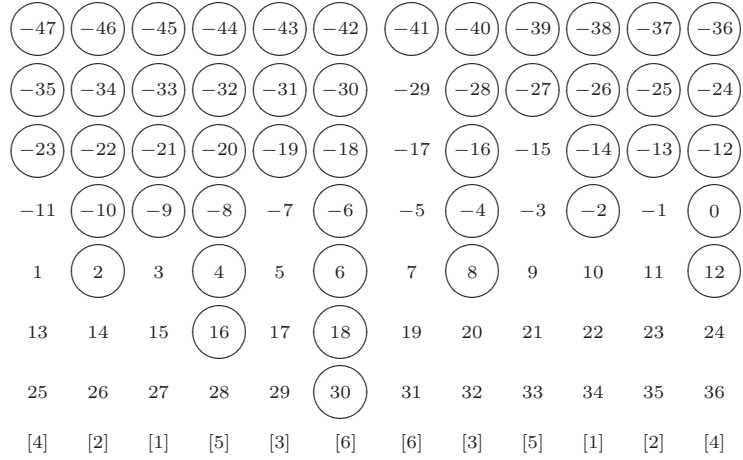

\[
\abacus 
{   
\ci{-47} &  \ci{-46} & \ci{-45} & \ci{-44} & \ci{-43} & \ci{-42} \,\,\, & \ci{-41} & \ci{-40} & \ci{-39} & \ci{-38} & \ci{-37} & \ci{-36} \\
\ci{-35} &  \ci{-34} & \ci{-33} & \ci{-32} & \ci{-31} & \ci{-30} \,\,\, & \nc{-29} & \ci{-28} & \ci{-27} & \ci{-26} & \ci{-25} & \ci{-24} \\
\ci{-23} &  \ci{-22} & \ci{-21} & \ci{-20} & \ci{-19} & \ci{-18} \,\,\, & \nc{-17} & \ci{-16} & \nc{-15} & \ci{-14} & \ci{-13} & \ci{-12} \\
\nc{-11} &  \ci{-10} & \ci{-9} & \ci{-8} & \nc{-7} & \ci{-6} \,\,\, & \nc{-5} & \ci{-4} & \nc{-3} & \ci{-2} & \nc{-1} & \ci{0} \\
\nc{1} &  \ci{2} & \nc{3} & \ci{4} & \nc{5} & \ci{6} \,\,\, & \nc{7} & \ci{8} & \nc{9} & \nc{10} & \nc{11} & \ci{12} \\
\nc{13} & \nc{14} & \nc{15} & \ci{16} & \nc{17} & \ci{18} \,\,\, & \nc{19} & \nc{20} & \nc{21} & \nc{22} & \nc{23} & \nc{24} \\
\nc{25} &  \nc{26} & \nc{27} & \nc{28} & \nc{29} & \ci{30} \,\,\, & \nc{31} & \nc{32} & \nc{33} & \nc{34} & \nc{35} & \nc{36} \\
\ & & & & & & & & & & & \ \ \\
[4] & [2] & [1] & [5] & [3] & [6] & [6] & [3] & [5] & [1] & [2] & [4] \\
}
\]
\caption{Abacus diagram for $\lambda = (0, 1, 4, 8, 14, 30)$.}\label{f:abex}
\end{figure}

Let $n = 6$.  Then $\lambda = (0, 1, 4, 8, 14, 30)$ is a lecture hall partition since 
\[ 0 \leq \frac{0}{1} \leq \frac{1}{2} \leq \frac{4}{3} \leq \frac{8}{4} \leq \frac{14}{5} \leq \frac{30}{6}. \]
Part of the abacus diagram for $\lambda$ is shown in Figure~\ref{f:abex}; the
unseen negative entries are all beads and the unseen positive entries are all
gaps.  The defining beads are $b_6 = 30$, $b_5 = 16$, $b_4 = 12$, $b_3 = 8$,
$b_2 = 2$, and $b_1 = -2$.  These beads lie in windows $5$, $3$, $2$, $2$, $1$,
and $0$, respectively.
\end{example}
 
For each $1 \leq i \leq n$, we say that entries in the column containing $b_i$
and in the dual column have {\bf class} $i$.  A position $p$ in the abacus
diagram for $\lambda$ is {\bf $i$-active} if $p$ lies weakly between position
$1$ and the position of the defining bead $b_i$ in reading order, and if the
class of $p$ is less than or equal to the class of $b_i$.  Then we can
summarize our construction as:

\begin{equation}\label{d:lhab}
\parbox{5in}{The abacus diagram for $\lambda$ is constructed by placing defining beads so
that there are $\lambda_i$ positions that are $i$-active, for each $1 \leq i \leq n$. }
\end{equation}

\begin{example}
In Figure~\ref{f:abex}, the classes of each column are indicated in brackets.  
The $4$-active positions are $12, 11, 10, 8, 5, 3, 2, 1$; there are $\lambda_4 = 8$ of these.
\end{example}

To describe the inverse construction, we will also consider arbitrary collections of
beads in the array.  We will say that such a collection of beads forms an
{\bf abacus diagram} if
\begin{itemize}
    \item No bead in any column is preceded in reading order by a gap in that
        column; when this condition holds, we say that the beads in the diagram are {\bf
        flush}.
    \item A bead occurs in position $j$ if and only if a gap occurs in position
        $1-j$ for all $j \in \mathbb{Z}$; when this condition holds, we say that
        the beads in the diagram are {\bf balanced}.
\end{itemize}
From the set consisting of the lowest bead in each column, the last $n$ of these
beads in reading order will be called the {\bf defining beads} of the abacus
diagram.  Observe that no two defining beads lie in dual columns, and that any
such set of defining beads determines a unique balanced flush abacus.

It is straightforward to verify that the abacus diagrams produced from lecture
hall partitions are abacus diagrams as defined in the preceding paragraph.
Moreover, we can recover a lecture hall partition from an arbitrary abacus
diagram by counting the number of $i$-active positions prior to each defining
bead in the diagram.

We claim that this is a bijection.

\begin{theorem} \label{the:LecAbaCor}
The lecture hall partitions are in bijection with abacus diagrams via the
constructions given above.
\end{theorem}
\begin{proof}
Composing the constructions, in either order, recovers the original object.
Hence, it suffices to prove that the inequalities defining the lecture hall partitions
are equivalent to the conditions defining the abacus diagrams.

For all $i$, we have
\[ \frac{\lambda_i}{i} \leq \frac{\lambda_{i+1}}{i+1} \text{ \ \ \ \ if and only if \ \ \ \ } \lambda_{i} \leq \lambda_{i+1} - \frac{\lambda_{i+1}}{i+1},  \]
which is equivalent to
\begin{equation}\label{e:lhineq}
\lambda_{i} \leq \lambda_{i+1} - \lceil{\frac{\lambda_{i+1}}{i+1}}\rceil,
\end{equation}
since the parts of $\lambda$ must be integers.

Under the correspondence (\ref{d:lhab}), each positive window prior to the
window containing the $(i+1)$st defining bead will have exactly $i+1$ positions
that are ($i+1$)-active.  Therefore, $\lceil{\frac{\lambda_{i+1}}{i+1}}\rceil$
represents the window containing the $(i+1)$st defining bead.  Hence, the
inequality in (\ref{e:lhineq}) means that the maximum number of $i$-active
positions is the number of ($i+1$)-active positions minus one position from each
positive window up to and including the window containing the $(i+1)$st defining
bead.  This difference is equivalent to the construction we have given, in which
the entries of class $i+1$ and all higher classes are ignored when placing $b_i$
so that there are $\lambda_i$ positive $i$-active positions.  In particular, the
defining beads $b_1, b_2, \ldots, b_n$ occur in the abacus diagram in reading
order.
\end{proof}

\bigskip
\section{Bounded partitions from abacus diagrams} \label{s:abbound}

We say that a partition $\boldsymbol{p}$ is {\bf bounded} if all of its parts
are at most $2n$ and those parts less than or equal to $n$ are distinct.  In
contrast to the lecture hall partitions, these partitions are straightforward to
enumerate:  if we let $|\boldsymbol{p}|$ denote the sum of the parts of
$\boldsymbol{p}$ then we obtain the generating function 
\[ \sum_{\text{ bounded partitions $\boldsymbol{p}$}}
x^{|\boldsymbol{p}|} = \frac{(1+x)(1+x^2)\cdots(1+x^n)}{(1-x^{n+1})(1-x^{n+2}) \cdots
(1-x^{2n})} = \prod_{i=1}^{n}\frac{1}{1-x^{2i-1}}. \]

We claim that each abacus diagram corresponds to a unique bounded partition.  
Consider an abacus diagram whose positive beads occur in positions $\{\dot{b}_1,
\dot{b}_2, \ldots, \dot{b}_k\}$; note that these are determined by the defining
beads, but we include all positive beads in this list.  We form the
partition $\boldsymbol{p}$ whose distinct parts consist of the positions of
those beads lying in the first window of the array, and for every positive bead
$\dot{b}_i$ lying outside the first window we include a part $\boldsymbol{p}_i$ of size
\[ \#( \, \textup{gaps between} \, \dot{b}_i-2n \, \textup{and} \, \dot{b}_i \, ) + 1. \]

\begin{example}
In Figure~\ref{f:abex}, we have beads in positions $2$, $4$, and $6$ lying in the
first window of the abacus diagram, so these are the distinct parts less than or
equal to $n = 6$.  The bead in position $8$ has $6$ gaps lying between itself
and the bead in position $-4$.  Similarly, there are $7$ gaps lying between beads
$12$ and $0$; $8$ gaps lying between beads $16$ and $4$; $8$ gaps lying between
beads $18$ and $6$; and $11$ gaps lying between beads $30$ and $18$.  
Therefore, the corresponding bounded partition is
\[ \boldsymbol{p} = (2, 4, 6, 7, 8, 9, 9, 12). \]
\end{example}

\begin{theorem} \label{the:BouAbaCor}
The abacus diagrams are in bijection with the bounded partitions via the
construction above.
\end{theorem}
\begin{proof}
We first show that the construction is well-defined. Clearly if $\dot{b}_i$ 
is in the first window then its corresponding part is between $1$ and $n$. 
Otherwise, $\dot{b}_i$ is a bead in a position greater than $n$, and we have 
that $\dot{b}_i-2n$ is also a bead because abacus diagrams are
flush.  There is one other position between $\dot{b}_i-2n$ and $\dot{b}_i$ that
is in the same class as $\dot{b}_i$, and this position must be a gap since
abacus diagrams are balanced.  The other $2n-2$ positions all belong to the
other $n-1$ classes.  Each class has precisely two positions, at least one of
which is a gap since abacus diagrams are balanced.  Hence, the number of gaps
between $\dot{b}_i-2n$ and $\dot{b}_i$ must be at least $n$.  On the other hand,
since there are $2n-1$ positions lying strictly between $\dot{b}_i-2n$ and
$\dot{b}_i$, there can be at most $2n-1$ gaps between them.  Hence, each part
$\boldsymbol{p}_i$ that we append satisfies $n+1 \leq \boldsymbol{p}_i \leq 2n$,
as required.  Thus, $\boldsymbol{p}$ is a composition of parts between $1$ and
$2n$ with those parts between $1$ and $n$ being distinct.  The flush condition
also implies that the number of gaps between $\dot{b}_i-2n$ and $\dot{b}_i$ is
increasing as a function of $\dot{b}_i$'s position.  Hence, if we append the
parts $\boldsymbol{p}_i$ following the reading order of the beads $\dot{b}_i$,
then $\boldsymbol{p}$ will be sorted increasingly so $\boldsymbol{p}$ is a
bounded partition.

Next, we give the inverse construction.  To encode a bounded partition
$\boldsymbol{p} = (\boldsymbol{p}_1, \cdots, \boldsymbol{p}_s,
\boldsymbol{p}_{s+1}, \cdots, \boldsymbol{p}_t)$, where $1 \leq \boldsymbol{p}_1
< \boldsymbol{p}_2 < \cdots < \boldsymbol{p}_s \leq n$ are the distinct parts,
begin by placing beads in positions $\boldsymbol{p}_1, \boldsymbol{p}_2, \ldots,
\boldsymbol{p}_s$ and leave all other positions in window $1$ as gaps.  Next,
place beads and gaps in window $0$ by leaving the positions
$1-\boldsymbol{p}_1, 1-\boldsymbol{p}_2, \ldots, 1-\boldsymbol{p}_s$ as gaps and
assigning beads to all other positions in window $0$.

To facilitate the rest of the construction, we say that a position $j$ in an
abacus diagram is {\bf supported} if the position $j-2n$ is a bead, and we say
that $j$ is {\bf unsupported} otherwise.  Having placed the first $i-1$ beads so
that each bead accurately encodes a part of the bounded partition, we claim that
there is a unique position for a new bead $\dot{b}_i$ such that there exist
exactly $\boldsymbol{p}_i - 1$ gaps between $\dot{b}_i$ and $\dot{b}_i - 2n$,
and so that the resulting abacus diagram remains flush.

To see this, imagine placing a new positive bead $\dot{b}_i$ in the position just
after the last bead $\dot{b}_{i-1}$ in reading order, or at position $n+1$ if $i =
s+1$.  Then the number of gaps between $\dot{b}_i$ and $\dot{b}_i - 2n$ is exactly
$\boldsymbol{p}_{i-1} - 1$, or simply $n$ if $i = s+1$.  Next, consider moving
$\dot{b}_i$
forward in reading order one entry at a time.  Each time we pass an unsupported
position $j$, we lose one gap from position $j-2n$ but we gain a gap at position
$j$, so the number of gaps between $\dot{b}_i$ and $\dot{b}_i - 2n$ is unchanged.  As we
pass a supported position $j$, we only gain the gap at position $j$ so the
number of gaps between $\dot{b}_i$ and $\dot{b}_i - 2n$ increases by $1$.

In order to both create the correct number of gaps and to have a flush abacus,
we must therefore place $\dot{b}_i$ at the $(\boldsymbol{p}_{i} -
\boldsymbol{p}_{i-1}+1)$st next supported position after $\dot{b}_{i-1}$, or at the
$(\boldsymbol{p}_{i} - n)$th next supported position after position $n$ in the
case that $i = s+1$.  Since the number of supported positions between $\dot{b}_{i-1}$
and $\dot{b}_{i-1}+2n$ remains equal to $2n - \boldsymbol{p}_{i-1} + 1$, this is
always possible.  This construction determines the bead/gap status of every
position greater than or equal to $-n$, and we complete the construction by
forming the unique balanced abacus that agrees with these entries.
\end{proof}

\begin{example}\label{ex:2}
Given the bounded partition $\boldsymbol{p} = (2,4,6,7,8,9,9,12)$, we form the partial abacus
diagram consisting of the distinct parts $\boldsymbol{p}_1 = 2$,
$\boldsymbol{p}_2 = 4$, and $\boldsymbol{p}_3 = 6$:
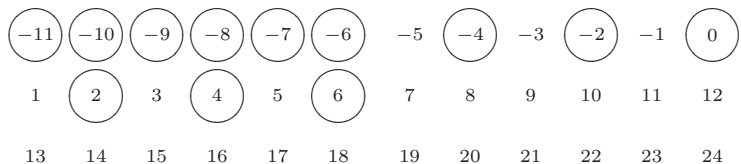
\[
\abacus 
{   
\ci{-11} &  \ci{-10} & \ci{-9} & \ci{-8} & \ci{-7} & \ci{-6} \,\,\, & \nc{-5} & \ci{-4} & \nc{-3} & \ci{-2} & \nc{-1} & \ci{0} \\
\nc{1} &  \ci{2} & \nc{3} & \ci{4} & \nc{5} & \ci{6} \,\,\, & \nc{7} & \nc{8} & \nc{9} & \nc{10} & \nc{11} & \nc{12} \\
\nc{13} & \nc{14} & \nc{15} & \nc{16} & \nc{17} & \nc{18} \,\,\, & \nc{19} & \nc{20} & \nc{21} & \nc{22} & \nc{23} & \nc{24} \\
}
\]
The supported positions are $8$, $10$, $12$, $14$, $16$ and $18$.  These
positions would correspond to bounded parts of size $7$, $8$, $9$, $10$, $11$
and $12$, respectively, as we can see by counting the number of gaps between
each entry and the corresponding entry in the previous row.  Since
$\boldsymbol{p}_4 = 7$, we must place the next bead in position $8$, obtaining:
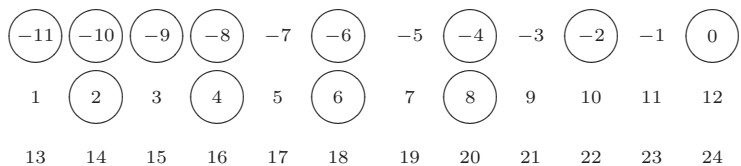
\[
\abacus 
{   
\ci{-11} &  \ci{-10} & \ci{-9} & \ci{-8} & \nc{-7} & \ci{-6} \,\,\, & \nc{-5} & \ci{-4} & \nc{-3} & \ci{-2} & \nc{-1} & \ci{0} \\
\nc{1} &  \ci{2} & \nc{3} & \ci{4} & \nc{5} & \ci{6} \,\,\, & \nc{7} & \ci{8} & \nc{9} & \nc{10} & \nc{11} & \nc{12} \\
\nc{13} & \nc{14} & \nc{15} & \nc{16} & \nc{17} & \nc{18} \,\,\, & \nc{19} & \nc{20} & \nc{21} & \nc{22} & \nc{23} & \nc{24} \\
}
\]
Now the supported positions are $10$, $12$, $14$, $16$, $18$ and $20$.  These
correspond to bounded parts of size $7$, $8$, $9$, $10$, $11$ and $12$,
respectively.  Since $\boldsymbol{p}_5 = 8$, we must place the next bead in
position $12$.  Continuing in this fashion, and then setting the bead/gap status
of the negative entries to balance with the positive entries we have specified,
we obtain the abacus in Figure~\ref{f:abex}.
\end{example}

\bigskip
\section{The bijections preserve the sums of parts} \label{s:bijpres}

Fix an abacus diagram on $2n$ columns.  We know from the previous sections that
there is a unique lecture hall partition $\lambda$ with $n$ parts and a
corresponding bounded partition $\boldsymbol{p}$. To show that the lecture hall
partitions are enumerated by the same generating function as the bounded
partitions, it suffices to show that $|\lambda| = |\boldsymbol{p}|$. 
In the running example from Figure~\ref{f:abex}, we have
\[ |\lambda| = 0+1+4+8+14+30 = 57 = 2+4+6+7+8+9+9+12 = |\boldsymbol{p}|. \]

The proof presented here is a straightforward induction proof, inducting on the
number of positive beads lying beyond position $n$ in an abacus.  It does need
a preliminary technical lemma. As before, the terms $b_i$ are (positions of)
the defining beads for class $i$ in a given abacus and, by
Theorem~\ref{the:LecAbaCor}, if $i < k$ then $b_i < b_k$.  Let $c(a)$ denote
the class of position $a$, and $\{  a < b : c(a) = k \}$ be the set of positive
positions less than $b$ of class $k$. 

\begin{lemma} \label{lem:ClassBead}
Fix an abacus diagram with defining beads $b_1, b_2, \ldots, b_n$.
If $i < k$ with $b_k - b_i < 2n$ then $ \# \{  a < b_k : c(a) = i \} - \# \{  a < b_i : c(a) = k \} = 1$. 
\end{lemma}
\begin{proof}
Denote the window that contains $b_i$ as the $\omega_i$-th window. Since $b_k$ comes after $b_i$ in reading order 
then in every window previous to the $\omega_i$-th window there is exactly one position of class $i$ and 
one of class $k$. On the other hand, since $b_k - b_i < 2n$ then $b_k$ is in one of the 
$\omega_i$-th window, $(\omega_i + 1)$-th window or the $(\omega_i + 2)$-th window. 
Consequently, the difference $ \# \{  a < b_k : c(a) = i \} - \# \{  a < b_i : c(a) = k \}$ can be 
restricted to those positions in the $\omega_i$-th, $(\omega_i + 1)$-th and $(\omega_i + 2)$-th 
windows. In the expression below we only count positions in these three windows. 

The difference $ \# \{  a < b_k : c(a) = i \} - \# \{  a < b_i : c(a) = k \}$ can be expanded as 
\begin{align*}
&(\underbrace{\# \{  a < b_i : c(a) = i \}}_{=0} - \# \{  a < b_i : c(a) = k \}) \, \\ 
+ \, &\underbrace{\# \{  a = b_i : c(a) = i \}}_{= 1} \, + \, \#  \{  a : b_i < a < b_k , \, c(a) = i \}.
\end{align*}

{\bf Case (i): the position of class $k$ in $\omega_i$ occurs before $b_i$}.  In
this case $b_k$ must be in the class $k$ position of either the
$(\omega_i+1)$-th window or the $(\omega_i+2)$-th window, the latter window
having the same class positions as $\omega_i$, in the former reversed.  Either
way the class $i$ position in the $(\omega_i+1)$-th window is the only position
in the set  $\{  a : b_i < a < b_k , \, c(a) = i \}$. Also, $\{  a < b_i : c(a)
= k \}$ has only one element, the class $k$ position in window $\omega_i$.
Hence, we have $\# \{  a < b_k : c(a) = i \} - \# \{  a < b_i : c(a) = k \} =
(0-1)+1+1 = 1$. 

{\bf Case (ii): the position of class $k$ in $\omega_i$ occurs after $b_i$}.  In
a similar fashion to Case (i), we clearly have $\# \{  a < b_i : c(a) = k \} =
0$. Since $b_k$ is either in a position after $b_i$ in the $\omega_i$-th
window or in a position before the class $i$ position in the $(\omega_i+1)$-th
window, the set $\{  a : b_i < a < b_k , \, c(a) = i \}$ is empty and $ \# \{  a
< b_k : c(a) = i \} - \# \{  a < b_i : c(a) = k \} = (0 -0)+1+0 = 1$.
\end{proof}

\begin{theorem} \label{the:OneNorm}
For every abacus diagram, the corresponding lecture hall partition $\lambda$ and
the corresponding bounded partition $\boldsymbol{p}$ satisfy $|\lambda| = |\boldsymbol{p}|$. 
\end{theorem}
\begin{proof}
We will prove the statement by induction on the number of positive beads $t$
lying beyond the first window in an abacus.  If $t = 0$ then the bounded partition
$\boldsymbol{p}$ corresponding to the abacus contains only distinct parts.
Since all the parts of the lecture hall partition $\lambda$ correspond to
positions in the abacus between $1$ and $n$, we have that $\lambda =
\boldsymbol{p}$.

Next, suppose we have an initial abacus with $t-1$ positive beads lying beyond
position $n$, and let us assume our inductive hypothesis that $|\lambda| =
|\boldsymbol{p}|$ for this initial abacus. Let us call this the $(t-1)$-abacus,
the prefix representing the number of positive beads lying beyond position $n$
in the abacus. Placing an additional positive bead in the $(t-1)$-abacus to
create a $t$-abacus diagram that is balanced and flush means that we can only
place a bead directly below an already existing defining bead. Assume this bead
is $b_i$, the defining bead of class $i$ in the $(t-1)$-abacus, and so the new
bead is in position $b_i + 2n$. Without loss of generality, we can assume that
$b_i+2n$ is the last bead in reading order in the $t$-abacus -- if it were not
we could remove the last bead in reading order to attain another abacus with
$t-1$ beads and assume the induction hypothesis on this abacus. All other
positions remain as beads or gaps as in the initial abacus but note that the
classes of the columns have changed. In particular, 
\begin{enumerate}
 \item[\textup{(a)}] the bead $b_i$ that was the defining bead of class $i$ in the $(t-1)$ abacus it is 
	  now of class $n$. It is no longer a defining bead, rather $b_i+2n$ is the defining 
	  bead of class $n$ in the $t$-abacus.  As a consequence, positions of class $i$ in the $(t-1)$-abacus 
	  are of class $n$ in the $t$-abacus; 
 \item[\textup{(b)}] the defining beads $b_{i+1}, b_{i+2}, \ldots, b_{n-1}, b_n$ in the $(t-1)$-abacus all lie between 
	  $b_i$ and $b_i+2n$ in reading order in the $t$-abacus. They remain defining beads in the $t$-abacus 
	  but their classes are now shifted down by $1$. That is, $b_k$ is of class $k-1$ in the $t$-abacus for all 
	  $i+1 \leq k \leq n$.  As a consequence, positions of class $k$ in the $(t-1)$-abacus 
	  are of class $k-1$ in the $t$-abacus; 
 \item[\textup{(c)}] the defining beads $b_1,\ldots,b_{i-1}$ of the $(t-1)$-abacus are all in positions less 
	  than $b_i$ and are defining beads for the same respective classes in the $t$-abacus.
\end{enumerate}
Let $\lambda^*$ and $\boldsymbol{p}^*$ denote the lecture hall and bounded partitions respectively of the $t$-abacus. 
Note that $\boldsymbol{p}^* = \boldsymbol{p}+p^*_t$ where the part $p_t^*$ is created by the new bead 
$b_i+2n$. Since the beads $b_{i+1}, \ldots, b_n$ are all strictly between $b_i$ and $b_i+2n$ and since all other 
defining beads are less than $b_i$ then the number of gaps between $b_i$ and $b_i+2n$ is $(2n-1)- (n-i) = n+i-1$. 
Hence the part $p_t^* = (n+i-1)+1 = n+i$ and $|\boldsymbol{p}^*| = |\boldsymbol{p}| + (n+i)$. 

All that remains to show is that 
$|\lambda_1^* + \lambda_2^* + \cdots + \lambda_n^*| 
= |\lambda_1 + \lambda_2 + \cdots + \lambda_n| + (n+i)$. We will do this by writing each $\lambda_k^*$ 
in terms of $\lambda_k$. 
Recall that for a given abacus, $\lambda_k$ is the sum of the $k$-active positions which, in the notation 
of Lemma~\ref{lem:ClassBead}, can be written as 
$\lambda_k = \sum_{j=1}^k{ \# \{  a \leq b_k : c(a) = j \} }$. 

By (c) above, $\lambda_k^* = \lambda_k$ for all $k$ less than $i$. By (a) above, the largest part from the 
$t$-abacus is 
\[ \lambda_n^* \, = \, \lambda_i + 2n +  \# \{  a < b_i : c(a) > i \}  \, = \, \lambda_i + 2n +  
\sum_{k=i+1}^n \# \{  a < b_i : c(a) = k \} \]
where $c(a)$ refers to the class of position $a$ in the $(t-1)$-abacus. By (b) above, for each 
$k = i, \ldots, n-1$ we have 
\[ \lambda_k^* \, = \, \lambda_{k+1} - \# \{  a < b_k : c(a) = i \} \]
where once again $c(a)$ refers to the class of position $a$ in the $(t-1)$-abacus.

This implies that 
\[ |\lambda^*| - |\lambda| \, = \, 2n - 
\sum_{k=i+1}^n 
\left(
\underbrace{\# \{  a < b_k : c(a) = i \} -   \# \{  a < b_i : c(a) = k \}}_{=1 \, \textup{ by Lemma~\ref{lem:ClassBead}}}
\right) \]
and so $|\lambda^*| - |\lambda| = 2n - (n-i) = n+i$ as claimed. 
\end{proof}

\begin{example}\label{ex:3}
Consider the case of $t = 4$ for our running example in which we add a bead to
position $18$:
\[
\abacus 
{   
\ & \ci{-35} &  \ci{-34} & \ci{-33} & \ci{-32} & \ci{-31} & \ci{-30} \,\,\, & \ci{-29} & \ci{-28} & \ci{-27} & \ci{-26} & \ci{-25} & \ci{-24} \\
\ & \ci{-23} &  \ci{-22} & \ci{-21} & \ci{-20} & \ci{-19} & \ci{-18} \,\,\, & \nc{-17} & \ci{-16} & \nc{-15} & \ci{-14} & \ci{-13} & \ci{-12} \\
\ & \nc{-11} &  \ci{-10} & \ci{-9} & \ci{-8} & \nc{-7} & \ci{-6} \,\,\, & \nc{-5} & \ci{-4} & \nc{-3} & \ci{-2} & \nc{-1} & \ci{0} \\
\ & \nc{1} &  \ci{2} & \nc{3} & \ci{4} & \nc{5} & \ci{6} \,\,\, & \nc{7} & \ci{8} & \nc{9} & \nc{10} & \nc{11} & \ci{12} \\
\ & \nc{13} & \nc{14} & \nc{15} & \ci{16} & \nc{17} & {\color{blue} \ci{18}} \,\,\, & \nc{19} & \nc{20} & \nc{21} & \nc{22} & \nc{23} & \nc{24} \\
\ & \ & & & & & & & & & & & \ \ \\
t-1: & [5] & [2] & [1] & [6] & [4] & [3] & [3] & [4] & [6] & [1] & [2] & [5] \\
\ & \ & & & & & & & & & & & \ \ \\
t: & [4] & [2] & [1] & [5] & [3] & [6] & [6] & [3] & [5] & [1] & [2] & [4] \\
}
\]
Then, the $(t-1)$-abacus corresponds to
\[ \lambda = (0,1,3,6,10,16) \ \ \text{ and } \ \ \boldsymbol{p} =
(2,4,6,7,8,9), \]
and the $t$-abacus corresponds to 
\[ \lambda^{*} = (0,1,6-2,10-2,16-2,3+12+3) = (0,1,4,8,14,18) \ \ \text{ and } \ \ \boldsymbol{p}^* = (2,4,6,7,8,9,9). \]
Note that 
\[ |\lambda^*| - |\lambda| = (12+3)+(-2)+(-2)+(-2) = 9 = |\boldsymbol{p}^*| - |\boldsymbol{p}|. \]
\end{example}

\bigskip
\section{The refined lecture hall theorem} \label{s:refined}

Given a lecture hall partition $\lambda= (\lambda_1, \lambda_2, \ldots, \lambda_n)$,  
let 
$
\lceil \lambda \rceil 
:=  
(\lceil \frac{\lambda_1}{1} \rceil, \lceil \frac{\lambda_2}{2} \rceil, \ldots, \lceil \frac{\lambda_n}{n} \rceil)
$ 
and let $o(\lceil \lambda \rceil)$ equal the number of odd parts of $\lceil \lambda \rceil$. 
In \cite{BME3} the following refinement of the Lecture Hall Theorem was shown.

\begin{theorem} \label{t:refined}
\textup{\bf (The Refined Lecture Hall Theorem)}
\ \ We have 
\begin{equation}\label{e:rgt}
\displaystyle 
\sum_{\lambda}x^{|\lambda|} u^{|\lceil \lambda \rceil|} v^{|o(\lceil \lambda \rceil)|} 
\, = \, 
\frac{(1+uvx)(1+uvx^2)\cdots(1+uvx^n)}{(1-u^2x^{n+1})(1-u^2x^{n+2}) \cdots (1-u^2x^{2n})}
\end{equation}
where the sum is taken over all lecture hall partitions $\lambda$ with $n$ parts and $|\lambda| = \sum_{i=1}^n \lambda_i$. 
 \end{theorem}
\begin{proof}
We claim that our bijections via abacus diagrams prove this refined version as
well.  Note that the specialization $u=v=1$ yields the Lecture Hall Theorem, and
so all we need to prove is the following: 
\begin{enumerate}
 \item[(a)] Every part of ${\bf p}$ in $\{ n+1, n+2, \ldots, 2n \}$ contributes $+2$ to the weight of $\lceil \lambda \rceil$. 
 \item[(b)] Every part of ${\bf p}$ in $\{ 1, 2, \ldots, n \}$ contributes $+1$ to the weight of $\lceil \lambda \rceil$. 
 \item[(c)] The number of parts of ${\bf p}$ in $\{ 1, 2, \ldots, n \}$ equals the number of odd parts of $\lceil \lambda \rceil$. 
 \end{enumerate}
The proof of Theorem~\ref{the:BouAbaCor} told us that every bead in a window $k \geq 2$ corresponded to a part $p$ of ${\bf p}$
with $n+1 \leq p \leq 2n$ and that every bead in the first window corresponded to a ``small'' part, $1 \leq p \leq n$ in ${\bf p}$.  
Recall also from the proof of Theorem~\ref{the:LecAbaCor}, that we labeled the window that contains the defining bead $b_i$ 
as the $\omega_i$-th window and that $\omega_i$ equals $\lceil \frac{\lambda_i}{i} \rceil$. 
With this in mind, the conditions (a)-(c) respectively are equivalent to the following conditions on the abacus diagram:  
\begin{enumerate}
 \item[(a')] Every bead in a window $k \geq 2$ contributes $+2$ to $\sum_{i=1}^n \omega_i$. 
 \item[(b')] Every bead in the first window contributes $+1$ to $\sum_{i=1}^n \omega_i$. 
 \item[(c')] The number of beads in the first window equals the number of odd $\omega_i$'s. 
\end{enumerate}
Since each window consists of $n$ positions, one position for each class $1 \leq i \leq n$, 
then $\omega_i$ can alternatively be expressed as  
$$
\omega_i = \#(\textup{positive positions} \, \leq b_i \, \textup{of class} \, i). 
$$
Suppose that there is a bead of class $i$ in the $k$-th window. Then, by the balanced condition, the 
class $i$ position in the $(k-1)$-th window is a gap and, by the flush condition, the class $i$ position in the 
$(k-2)$-th window is a bead, and so on. Consequently, the number of positive class $i$ positions $\leq b_i$ 
can be written in terms of beads:  
\[
\omega_i = \begin{cases}
    \parbox{2.5in}{$2 \cdot$ \#(beads of class $i$ in a window $k \geq 2$) + \#(beads of class $i$ in the first window)} & \textup{if} \, \omega_i \, \textup{is odd} \\ 
    & \\
    \parbox{2.5in}{$2 \cdot$ \#(beads of class $i$ in a window $k \geq 2$)} & \textup{if} \, \omega_i \, \textup{is even} \\ 
\end{cases}
\]
Since every bead is of one and only class then (a') and (b') are satisfied. Finally, by the flush condition 
the beads in the first window are either defining beads themselves or they are supported below by a defining bead. 
Each of these defining beads must live in an odd window and so (c') is satisfied.  
\end{proof}

\bigskip
\section{Conclusions} \label{s:conclusions}

Although our exposition has been self-contained, the combinatorics we have
developed is relevant to the affine Weyl group $\widetilde{C}_n$ and compares
favorably with the earliest proof of the lecture hall theorem that relies on
Bott's formula \cite{BME}.  In this section, we briefly review these
connections.

Recall that a Coxeter group is a group $W$ with a certain presentation in terms
of generators $s_0, s_1, \ldots, s_n$, each of which is an involution, such that
the only relations in $W$ arise as a consequence of imposing dihedral subgroup
structures on the subgroups generated by each pair of generators.
For each group element $w$, we let $\ell(w)$ denote the minimal length of any
expression for $w$ in the generators $s_0, s_1, \ldots, s_n$.  A fundamental
enumeration problem for any Coxeter group $W$ is to describe the generating
function $\sum_{w \in W} t^{\ell(w)}$.  When $W$ is a finite or affine Weyl
group, this problem has applications to algebraic geometry and representation
theory.

Since any subset $J$ of the generators will generate a subgroup $W_J$ of $W$, we
may consider the cosets of $W_J$ in $W$.  The set of these cosets is often
denoted $W/W_J$.  It turns out that each coset contains a unique element of
minimal length, and we denote the set of these minimal length coset
representatives by $W^J$.  If we abuse notation to let $X(t)$ denote $\sum_{w
\in X} t^{\ell(w)}$ for any subset $X$ of $W$, then we obtain the factorization
\[ W(t) = W_J(t) W^J(t). \]
It follows from this that $W(t)$ is always a rational generating function that
can be computed inductively.

Bott's formula is an explicit description of $W^J$ when $W$ is an affine Weyl
group, and $J$ is the set of generators for the corresponding finite Weyl
subgroup.  It turns out that $W^J$ always takes the form $\prod_{i=1}^n
\frac{1}{1-t^{e_i}}$ where the $e_i$ are the ``exponents'' of the Weyl group;
see \cite{humphreys} for details.

In the case when $W = \widetilde{C}_n$ and $J = \{s_1, \ldots, s_n\}$, we happen
to obtain
\[ W^J(t) = \prod_{i=1}^n \frac{1}{1-t^{2i-1}}, \]
the same generating function as for restricted odd partitions or lecture hall
partitions.  This empirical fact is probably what led Bousquet-M\'elou and
Eriksson to their original proof of the lecture hall theorem.

In that proof, the authors explained this coincidence by realizing the Weyl
group $\widetilde{C}_n$ as a subgroup of permutations of the integers, using a
certain carefully developed embedding.  They provided a bijection between the
lecture hall partitions and these integer permutations.  Under this map, the sum
of the parts of the lecture hall partition corresponds with an inversion
statistic on the integer permutation $w$ that is known to be equivalent to
$\ell(w)$.

Our proof uses combinatorics that have been developed recently in
\cite{hanusa-jones} to generalize James' abacus model \cite{james-kerber} from
type $\widetilde{A}$ to the other affine types.  The abacus diagrams we have
described in the present paper are identical to those defined in
\cite{hanusa-jones}.  It is shown there that the abacus diagrams correspond to
elements $w \in W^J$, and that from the abacus diagram it is possible to read
off the bounded partitions that are known to have sum of parts equal to
$\ell(w)$ \cite[Proposition 7.4]{hanusa-jones}.  In fact, our work here together
with the results of \cite{hanusa-jones} could be viewed as an independent proof
of Bott's formula in type $\widetilde{C}$.

\bigskip
\section*{Acknowledgments}

This work was initiated during the summer of 2012 at a research experience for
undergraduates (REU) program at James Madison University mentored by the third
and sixth authors.  We thank JMU and Leonard Van Wyk for their support.  We
also thank Carla Savage and Matthias Beck for helpful conversations.  Finally,
we would like to acknowledge the anonymous referee for providing useful
references and comments on an earlier draft of this work.


\bibliographystyle{alpha}

\end{document}